\documentclass[12pt]{amsart}%
\usepackage{amssymb}
\usepackage{amsfonts}
\usepackage{amsmath}
\usepackage{amsthm}
\usepackage{hyperref}
\usepackage{psfrag, mathrsfs}
\usepackage{graphicx}%
\usepackage{overpic}
\usepackage{tikz}
\providecommand{\U}[1]{\protect\rule{.1in}{.1in}}
\newtheorem{theorem}{Theorem}
\newtheorem*{theorem*}{Theorem}

\newtheorem*{conjecture*}{Conjecture}
\newtheorem*{belief*}{General Belief}
\newtheorem{corollary}[theorem]{Corollary}

\newtheorem{definition}[theorem]{Definition}

\newtheorem{lemma}[theorem]{Lemma}
\newtheorem{notation}[theorem]{Notation}

\newtheorem{proposition}[theorem]{Proposition}

\newtheorem{question}[theorem]{Question}

\renewenvironment{proof}[1][Proof]{\noindent\textbf{#1.} }{\ \rule{0.5em}{0.5em}}

\newcommand{\Z}{\mathbb{Z}}
\newcommand{\R}{\mathbb{R}}
\newcommand{\Rf}{\mathbf{R}_{\floor{\cdot}}}
\newcommand{\Rc}{\mathbf{R}_{\ceil{\cdot}}}
\newcommand{\N}{\mathbb{N}}
\newcommand{\Q}{\mathbb{Q}}

\newcommand{\floor}[1]{\left\lfloor {#1}\right\rfloor}
\newcommand{\ceil}[1]{\left\lceil {#1}\right\rceil}

\newcommand{\proj}{\mathrm{proj}}
\renewcommand{\mod}{\ \mathrm{mod}\ }
\newcommand{\comm}{\operatorname{comm}}

\begin{document}

\title{\textbf{The Doorways Problem and Sturmian Words}
}
\author{Jason Siefken}
\address{Northwestern University, Evanston, IL 60208 USA} 
\email{siefkenj@math.northwestern.edu}
\date{\today}
\maketitle

\begin{abstract}
	The doorways problem considers adjacent parallel hallways of unit
	width each with a single doorway (aligned with integer lattice
	points) of unit width.  It then asks, what are the properties of
	lines that pass through each doorway?  Configurations of doorways
	closely correspond to Sturmian words, and so properties of these configurations
	may be lifted to properties of Sturmian words.
	This paper classifies the
	slopes of \emph{lines of sight}, lines that pass through each doorway,
	for both the case of a finite number of parallel hallways and an
	infinite number and their consequences for Sturmian words.  
	We then produce a metric on configurations with
	an infinite number of hallways that preserves the property of admitting
	a line of sight under limits.  Pulling back this metric to $\R$, we 
	produce the Baire metric under which the irrational numbers form a complete
	metric space. Pulling back this metric to the set of all Sturmian sequences,
	we show that the set of all Sturmian sequences is complete with this metric
	(unlike with the standard metric).
\end{abstract}

\section{The Doorways Problem}

	Imagine a series of $n+1$ infinitely long parallel walls spaced
	one unit apart, creating $n$ hallways.
	Further imagine that each wall
	has infinitely many doors of unit width, but that only one door
	per wall is open.

	Standing to one side of the hallways, you could imagine
	certain arrangements of open doors you could see through
	and certain arrangements you could not. This is precisely stated in the
	following definitions.

	\begin{definition}[Hallway]
		An \emph{$n$-hallway} is the set $H_n\subset \R^2$ defined by
		\[
			H_n = \bigcup_{i\in\{0,\ldots,n\}} \{i\}\times (\R\backslash D_i)
		\]
		where $D_i=(d_i,d_i+1)$ is an open interval of width one and left
		point $d_i\in \Z$.
		The set $D_i$ is called the $i$th \emph{doorway}.
	\end{definition}

	\begin{definition}[Line of Sight]
		\label{DefLineOfSight}
		Given an $n$-hallway $H_n$, we can \emph{see through} $H_n$ if
		there exists some line $\ell_{\alpha\beta}=\{(x,\alpha x+\beta):x\in\R\}$
		with slope $\alpha$ and $y$-intercept $\beta$ so that $\ell_{\alpha\beta}\cap H_n=\emptyset$.
		If $\ell_{\alpha\beta}\cap H_n = \emptyset$, we call $\ell_{\alpha\beta}$ a
		\emph{line of sight} and we say $H_n$ \emph{admits} the line of sight $\ell_{\alpha\beta}$.
		If $\ell_{\alpha\beta}$ is a line of sight and $\alpha\in\Q$, we call $\ell_{\alpha\beta}$
		a \emph{rational line of sight}.
	\end{definition}

	Note that Definition \ref{DefLineOfSight} captures the idea that ``no light is blocked'' by
	a hallway.  This could be equivalently phrased as \emph{an $n$-hallway $H_n$ with doorways $D_i$
	admits the line of sight $\ell_{\alpha\beta}$ if $\ell_{\alpha\beta}\cap (\{i\}\times D_i)\neq \emptyset$
	for all $i$}, which would capture the idea that ``light passed through every doorway.''  However, 
	upon the introduction of infinite hallways, the ``no light is blocked'' definition will be more useful.

	The doorways problem in general asks what types of
	$n$-hallways can be seen through, and what are the properties
	of lines of sight.  This question is closely related to rotation
	sequences, balanced sequences, and Sturmian sequences \cite{fogg, lothaire},
	and it is from this context that the following motivating question arises.
	
	\begin{question}\label{QuestRat}
		For an $n$-hallway $H_n$ that can be seen through, is there always a line
		of sight $\ell_{\alpha\beta}$ with slope $\alpha=\frac{p}{q}$ where
		$q\leq n$?
	\end{question}

\subsection{Connection to Sturmians}
	Sturmian sequences and Sturmian words have many equivalent definitions in terms of
	rotation sequences, billiard sequences, balanced words, complexity, and invariant
	measures \cite{fogg, lanford, lothaire}.  For the sake of brevity, we provide only
	two equivalent definitions.
	\begin{definition}[Complexity]
		For a sequence $x\in\{0,1\}^\N$, the \emph{complexity function} is
		\[ L_n(x)=
		\#\{\text{distinct subwords of }x\text{ of length }n\}.\]
	\end{definition}
	\begin{definition}[Periodic and Eventually Periodic]
		For a sequence $x\in\{0,1\}^\N$, let $(x)_i$ be the $i$th coordinate
		of $x$. The sequence $x\in\{0,1\}^\N$ is called \emph{periodic} if there
		exists $m>0$ so that $(x)_i=(x)_{i+m}$ for all $i$ and is called \emph{aperiodic}
		otherwise.  The sequence is called \emph{eventually periodic} if there exists $m>0$
		and some $I$ so that $(x)_i=(x)_{i+m}$ for all $i>I$.
	\end{definition}

	\begin{definition}[Sturmian Sequence]
		Let $x\in\{0,1\}^\N$.  The sequence $x$ is a \emph{Sturmian sequence} if it is
		periodic and satisfies $L_n(x)\leq n+1$ for all $n$ or if it satisfies
		$L_n(x)=n+1$ for all $n$.
		A
		\emph{Sturmian word} is a subword of a Sturmian sequence.
	\end{definition}

	A sequence $x\in\{0,1\}^\N$ satisfying $L_n(x)=n+1$ is always aperiodic and never
	eventually periodic.  Thus, an eventually periodic Sturmian sequence must be periodic.
	Hedlund and Morse \cite{morse} proved that for any $x\in\{0,1\}^\N$, $x$
	is eventually periodic if and only if there
	exists an $n$ such that $L_n(x) < n+1$.  Viewed this way, aperiodic
	Sturmian sequences are the aperiodic sequences of the lowest possible complexity.

	\begin{definition}[Rotation Sequence]
		For a pair $(\alpha,\beta)\in[0,1]\times \R$, the rotation sequences $s=\Rf(\alpha,\beta)\in\{0,1\}^\N$ and
		$s'=\Rc(\alpha,\beta)\in\{0,1\}^\N$ are the sequences whose $i$th coordinates are given by
		\[
			(s)_i = \floor{(i+1)\alpha +\beta}-\floor{i\alpha+\beta}
		\]
		and
		\[
			(s')_i = \ceil{(i+1)\alpha +\beta}-\ceil{i\alpha+\beta},
		\]
		where $\floor{\cdot}$ and $\ceil{\cdot}$ are the floor and ceiling functions,
		respectively. A sequence $x\in\{0,1\}^\N$ is called a \emph{rotation sequence} if $x=\Rf(\alpha,\beta)$
		or $x=\Rc(\alpha,\beta)$ for some $(\alpha,\beta)\in [0,1]\times\R$.
	\end{definition}

	As shown in \cite{fogg, lothaire}, a sequence $x\in\{0,1\}^\N$ is Sturmian if and only if it is a rotation
	sequence.  Further, every Sturmian word appears as the starting word of a rotation sequence (equivalently Sturmian
	sequence).

	\vspace{.5cm}
	Given an $n$-hallway $H_n$ with doorways $D_i=(d_i,d_i+1)$, there is a natural correspondence between
	$H_n$ and elements in $\Z^n$.  Namely, associate $H_n$ with the 
	$n$-word $(d_1-d_0,d_2-d_1,\ldots, d_n-d_{n-1})$
	given by the differences between positions of consecutive doorways.  Let $\Phi:\{\text{hallways}\}\to\{\text{words}\}$
	denote this correspondence.

	The question of whether a hallway admits a line of sight only depends on
	the relative placement of each doorway and is therefore translation invariant.  
	Thus $H_n$ admits a line of sight if and only if every hallway
	in $\{H_n':\Phi(H_n')=\Phi(H_n)\}$ admits a line of sight.

	Fix an $n$-hallway $H_n$ with initial doorway $D_0=(0,1)$ and suppose $\Phi(H_n)$ is a Sturmian word.
	Further suppose $\Phi(H_n)$ appears as the initial word for the rotation sequence $s=\Rf(\alpha,\beta)$ 
	and that
	$(i\alpha+\beta)\notin\Z$ for $0\leq i\leq n$.  We can now conclude that 
	\[
		D_i = (\floor{ i\alpha+\beta},\floor{i\alpha+\beta}+1)
	\]
	and  $H_n$ admits the line of sight $\ell_{\alpha\beta}$.  The converse of this statement also
	holds, and with the technical assumptions minimized, we get Theorem \ref{ThmSturmClassification}.

	\begin{theorem}\label{ThmSturmClassification}
		Let $\Psi_a:\{a,a+1\}^n\to\{0,1\}^n$ be the map that sends $a\mapsto 0$ and $(a+1)\mapsto 1$.
		The $n$-hallway $H_n$ admits a line of sight if and only if $\Phi(H_n)\in \{a,a+1\}^n$
		for some $a$ and $\Psi_a\circ \Phi(H_n)$ is a Sturmian word.
	\end{theorem}

	We will not prove Theorem \ref{ThmSturmClassification} in the context of Sturmian sequences,
	however, studying the hallway problem directly we will arrive at 
	equivalent results.  Theorem \ref{ThmSturmClassification} also gives context as to why 
	Question \ref{QuestRat} might be interesting.  

	Consider the following: given a finite Sturmian word $w$, is $w$ always contained in a
	periodic Sturmian word?  If so, what is the minimum period of such a word?  Translating from
	hallways to rotation sequences to Sturmian sequences, Question \ref{QuestRat} asks, ``Is
	a finite Sturmian word $w$ always contained in a periodic Sturmian sequence with period
	bounded by the length of $w$?''

	Studying $n$-hallways will provide a geometric way to answer this question.  Further,
	the extension of $n$-hallways to infinite hallways will 
	allow us to arrive at several results without the subtleties of working with Sturmian
	sequences or rotation sequences directly.  In particular, the distinction between \emph{aperiodic}
	and \emph{not eventually periodic}  and the need to include
	both $\floor{\cdot}$ and $\ceil{\cdot}$ (as in the definition of rotation sequences) is avoided.

\section{Answering the Question}

	As discussed earlier, the question of whether an $n$-hallway admits a line of sight
	is translation invariant.
	Thus, we will assume that all $n$-hallways satisfy $D_0=(0,1)$.
	Now, we will tackle the question of whether or not there exists lines of sight.

	\begin{definition}
		Let $\proj_\gamma:\R^2\to \R$ be \emph{parallel projection onto the $y$-axis}
		along a line of slope $\gamma$.  That is,
		\[
			\proj_\gamma(x,y) = y-\gamma x.
		\]
	\end{definition}

	\begin{proposition}
		\label{PropIntervalSight}
		If $H_n$ is an $n$-hallway that admits a line of sight, then there is
		an interval of slopes corresponding to lines of sight for $H_n$.
	\end{proposition}
	\begin{proof}
		Fix $H_n$, an $n$-hallway, and let $\ell_{\alpha\beta}$ be a line of sight.
		We now have $\proj_\alpha(\ell_{\alpha\beta})=\{\beta\}$.

		Let $D_i=(d_i,d_i+1)$ be the $i$th doorway of $H_n$, and let
		\[
			D=\bigcap_{0\leq i\leq n} \proj_\alpha(\{i\}\times D_i).
		\]
		Since $\ell_{\alpha\beta}$ is a line of sight, $\beta\in D\neq \emptyset$.  Since $D$ is a finite
		intersection of open intervals, $D=(d_l,d_r)$ is an open interval.  It directly follows
		that the ``tube'' $T=\proj_\alpha^{-1}(D)\cap([0,n]\times\R)$ safely passes through every
		doorway in $H_n$, and as a consequence any line contained in $T$ will be a line of sight.
		See Figure \ref{FigTube} for an example.

		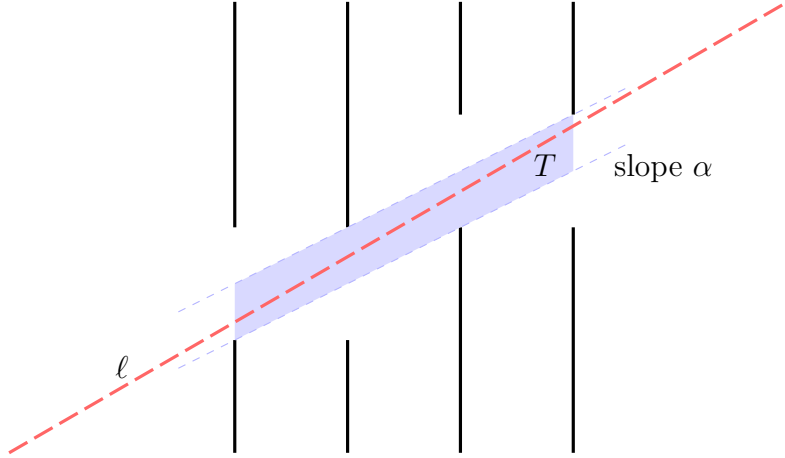
\begin{figure}[h!]
			\begin{center}

			\begin{tikzpicture}[scale=1.5]

			    \draw[very thick] (0,0) -- (0,1);
			    \draw[very thick] (0,2) -- (0,4);
			    
			    \draw[very thick] (1,0) -- (1,1);
			    \draw[very thick] (1,2) -- (1,4);
			    
			    \draw[very thick] (2,0) -- (2,2);
			    \draw[very thick] (2,3) -- (2,4);
			    
			    \draw[very thick] (3,0) -- (3,2);
			    \draw[very thick] (3,3) -- (3,4);

			    \path[fill=blue!15] (0,1) -- (3, 2.5) -- (3, 3) -- (0,1.5) -- cycle;
			    \draw[dashed, blue!30] (-.5,.75) -- (3.5,2.75);
			    \draw[dashed, blue!30] (-.5,1.25) -- (3.5,3.25);
			    \draw (2.75,2.75) node[yshift=-3mm] {$T$};
			    \draw (3,2.75) node[below right,xshift=4mm] {slope $\alpha$};
			    
			    \draw[dash pattern=on 8pt off 3pt, red!60, very thick] (-2,0) -- (4.9,4);
			    \draw (-1,.5) node[above,yshift=1mm] {$\ell$};
			\end{tikzpicture}
			\end{center}
			\caption{ \footnotesize 
				A hallway with four doorways and the slope-$\alpha$ tube
				$T=\proj_\alpha^{-1}(D)\cap([0,4]\times\R)$ along with a line
				of sight $\ell$ having slope greater than $\alpha$.
			}
			\label{FigTube}
		\end{figure}

		Since the width of $T$ is $n$ and the height of $T$ is $d_r-d_l>0$, we know there must be lines 
		of sight for every slope in the interval $(\alpha - \frac{d_r-d_l}{n}, \alpha + \frac{d_r-d_l}{n})$.
	\end{proof}

	\begin{corollary}
		\label{PropRationalSight}
		If $H_n$ is an $n$-hallway that admits a line of sight, then $H_n$
		admits a rational line of sight.
	\end{corollary}
	Since every interval of real numbers contains a rational number, Corollary \ref{PropRationalSight}
	follows immediately.

	The proof of Proposition \ref{PropIntervalSight} gives some insight into
	what types of hallways admit lines of sight. In particular, a hallway
	admits a line of sight if and only if $D^{\alpha} = \bigcap \proj_\alpha(\{i\}\times D_i)$
	is non-empty for some $\alpha$.

	Suppose $\alpha\in[0,1)$, and consider $1$-hallways.  Recall
		that we always assume $D_0=(0,1)$ is the first doorway.  Now, if $\ell_{\alpha\beta}$
	is a line of sight for a $1$-hallway
	$H_1$, because $0\leq \alpha<1$, there are only two possibilities for $D_1$.  Namely, $D_1=(0,1)$
	or $D_1=(1,2)$.  Since we are assuming $\ell_{\alpha\beta}$ is a line of sight
	for $H_1$, we can completely determine what $D_1$ is by the following procedure:
	Let $s_1 = \proj_\alpha(1,1)$.  
	If $\beta$, the $y$-intercept of $\ell_{\alpha\beta}$, satisfies $\beta<s_1$, then $D_1=(0,1)$.  If $\beta > s_1$ then $D_1=(1,2)$.
	See Figure \ref{FigDetermineNextDoorway} for an illustration. 

		\begin{figure}[h!]
			\begin{center}
			\begin{tikzpicture}[scale=2]
				
				\foreach \x in {0,...,1}
					\foreach \y in {-1,...,2}
					{
						\fill (\x,\y) circle (.3mm) node[below right, yshift=2mm] {\footnotesize $(\x,\y)$};
					}
			    \draw[very thick] (0,-1.3) -- (0,0);
			    \draw[very thick] (0,1) -- (0,2.3);

			    \foreach \y in {0.67,1.67}
			    {
				    \draw[dashed, blue!60, thin] (-1,{-1+\y}) -- (2,{0+\y});
			    }
			    \foreach \y in {1.33}
			    {
				    \draw[dash pattern=on 8pt off 3pt, red!60, very thick] (-1,{-1+\y}) -- (2,{0+\y});
			    }
			    \fill (0,.67) circle (.3mm) node[below right] {$s_1$};
			    \draw (-1,.9) node[above right] {slope $\alpha$};
			\end{tikzpicture}
			\begin{tikzpicture}[scale=2]
				
				\foreach \x in {0,...,2}
					\foreach \y in {-1,...,2}
					{
						\fill (\x,\y) circle (.3mm); 
					}
			    \draw[very thick] (0,-1.3) -- (0,0);
			    \draw[very thick] (0,1) -- (0,2.3);

			    \foreach \y in {0.67,1.67}
			    {
				    \draw[dashed, blue!60, thin] (-1,{-1+\y}) -- (3,{.33+\y});
			    }
			    \foreach \y in {1.33, 1, 2}
			    {
				    \draw[dash pattern=on 8pt off 3pt, red!60, very thick] (-1,{-1+\y}) -- (3,{.33+\y});
			    }
			    \fill (0,.67) circle (.3mm) node[below right,yshift=1mm] {$s_1$};
			    \fill (0,.33) circle (.3mm) node[below right,yshift=1mm] {$s_2$};
			    \fill (0,1.33) circle (.3mm) node[below right,yshift=2mm] {$s_2'$};
			\end{tikzpicture}
			\end{center}
			\caption{ \footnotesize 
				For a slope of $\alpha$, $s_1$ and $s_2$ divide $(0,1)$ into regions based on
				the doorways $D_1$ and $D_2$.
			}
			\label{FigDetermineNextDoorway}
		\end{figure}
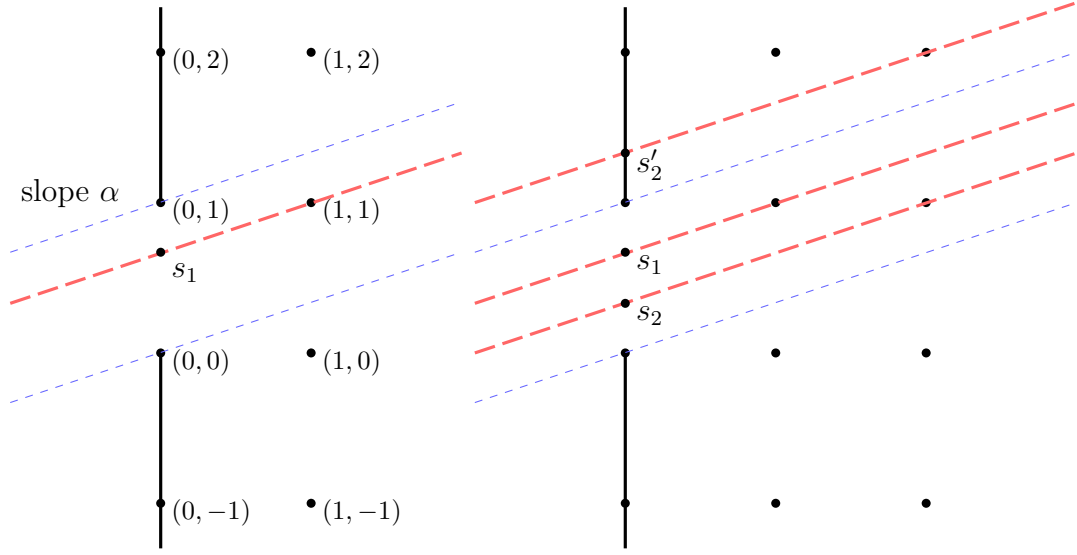

	Inductively, we may determine all possible hallways
	corresponding to lines of sight with a particular slope.  For illustration,
	consider $2$-hallways and assume we have a line of sight $\ell_{\alpha\beta}$.
	If $D_1=(0,1)$, we know $D_2=(0,1)$ or $(1,2)$.  
	To find out which, let $s_2=\proj_\alpha(2,1)$.  If $\beta < s_2$, $D_2=(0,1)$,
	and if $\beta > s_2$, $D_2=(1,2)$.  Similarly, if $D_1=(1,2)$, the options
	for $D_2$ are $(1,2)$ or $(2,3)$.  Now letting $s_2' = \proj_\alpha(2,2)$,
	$\beta < s_2'$ implies $D_2=(1,2)$ and $\beta > s_2'$ implies $D_2=(2,3)$.

	If we continue this process, we will notice that we always
	consider $s=\proj_\alpha(x_0,y_0)$ as a bifurcation point.  That
	is, $s$ allows us to decide which doorway must be open if we presuppose
	a certain line of sight.

	\begin{definition}
		Let $Y_{\alpha, n} = (0,1)\backslash\left(\bigcup_{i\leq n}\proj_\alpha (\{i\}\times \Z)\right)$
		and let $\mathcal Y_{\alpha, n}$ be the partition of $Y_{\alpha, n}$ consisting of its
		connected components.
	\end{definition}

	As the next proposition shows, $\mathcal Y_{\alpha, n}$ exactly classifies which
	sequence of doors must be open for a given line of sight.  In this respect, $\mathcal Y_{\alpha, n}$
	can be seen as generating an equivalence relation on intercepts of lines of sight with slope $\alpha$.

	\begin{proposition}
		\label{PropYPartitions}
		Fix $\alpha\in [0,1)$ and suppose $\ell_{\alpha\beta}$ is
			a line of sight for the $n$-hallway $H_n$ and $\ell_{\alpha\beta'}$ is a line of sight
			for the $n$-hallway $H_n'$.  Then, $\beta,\beta'\in Y$ for some 
			$Y\in \mathcal Y_{\alpha, n}$ if and only if $H_n=H_n'$.
	\end{proposition}

	\begin{proof}
		Notice that $H_n=H_n'$ if and only if the sequences of doorways for $H_n$ and $H_n'$ are the same.
		
		Suppose $H_n=H_n'$ and that $\ell_{\alpha\beta}$ and $\ell_{\alpha\beta'}$ are lines of
		sight for $H_n$ and $H_n'$, respectively.  By construction $Y=(\proj_{\alpha} H_n)^c=(\proj_{\alpha} H_n')^c
		\in \mathcal Y_{\alpha, n}$, and we must have $\beta,\beta'\in Y$.

		Now suppose that $\beta,\beta'\in Y$ for some $Y\in \mathcal Y_{\alpha, n}$.
		We will proceed by induction.

		The base case is clear.  $H_1=H_1'$, since $\mathcal Y_{\alpha, 1}$ precisely
		partitions the $y$-intercepts for lines of sight through $(D_0,D_1)=\Big((0,1), (0,1)\Big)$
		and $(D_0,D_1)=\Big((0,1), (1,2)\Big)$.

		Assume the proposition holds for $n-1$.  This means that $D_i=D_i'$
		for $i < n$.  Fix $a$ so that $D_{n-1}=D_{n-1}'=(a,a+1)$.
		Since $\alpha\in [0,1)$, there can only be two possibilities
		for $D_n$ or $D_n'$.  Namely, $(a,a+1)$ or $(a+1, a+2)$.
		Let $s=\proj_\alpha(n,a+1)$ and suppose $\ell_{\alpha\gamma}$ is
		a line of sight for $H_n$.  If $\gamma < s$, $D_n=(a,a+1)$
		and if $\gamma > s$, $D_n=(a+1,a+2)$.

		We complete the proof by noticing that $s$ lies on the boundary of a partition
		element of $\mathcal Y_{\alpha, n}$ (or completely outside the interval $(0,1)$).  
		Thus, if $\beta,\beta'\in Y$, we have
		$\beta,\beta'> s$ or $\beta,\beta' < s$, and so $D_n=D_n'$.
	\end{proof}

	Propositions like 
	Proposition \ref{PropYPartitions} can be extended to handle lines of sights with
	slopes in $\R$ without too much trouble, so we will mainly focus on lines of sights
	with slopes in $[0,1)$ to make our arguments simpler.

	\begin{corollary}
		\label{PropUniqueHallway}
		For a fixed line $\ell_{\alpha\beta}$, there is at most one $n$-hallway
		such that $\ell_{\alpha\beta}$ is a line of sight.
	\end{corollary}
	\begin{proof}
		This follows directly from an application of Proposition \ref{PropYPartitions}
		with $\beta=\beta'$.
	\end{proof}

	\begin{proposition}
		\label{PropNumberOfPartitionElements}
		For a fixed $\alpha$, the number of elements in the partition $\mathcal Y_{\alpha, n}$
		is at most $n+1$.
	\end{proposition}
	\begin{proof}
		First, notice that since $\proj_\alpha(x,y)-\proj_\alpha(x,y')=y-y'$, for a fixed $i$,
		$(0,1)\cap \proj_\alpha (\{i\}\times \Z)$ contains at most one point. This follows from
		the fact that integers are one unit apart and $(0,1)$ is an open interval of width one.

		Now, we may proceed by induction.  Clearly $\mathcal Y_{\alpha, 0}=\{(0,1)\}$ consists of one
		interval.  Suppose $\mathcal Y_{\alpha, n-1}$ consists of no more than $n$ intervals.  $\mathcal Y_{\alpha, n}$
		can be obtained from $\mathcal Y_{\alpha, n-1}$ by slicing the partition elements of $\mathcal Y_{\alpha, n-1}$
		by the points in $(0,1)\cap \proj_\alpha (\{n\}\times \Z)$.  But, there is at most one
		point in $(0,1)\cap \proj_\alpha (\{n\}\times \Z)$ and so at most one interval in $\mathcal Y_{\alpha, n-1}$ could
		be sliced into two intervals.  Thus the number of elements in $\mathcal Y_{\alpha, n}$ cannot exceed $n+1$.
	\end{proof}

	\begin{corollary}\label{PropNPlus1}
		For a fixed $\alpha$, the number of distinct $n$-hallways having $D_0=(0,1)$
		and admitting a line of sight of slope $\alpha$ is at most $n+1$.
	\end{corollary}
	\begin{proof}
		From Proposition \ref{PropYPartitions}, $\mathcal Y_{\alpha, n}$ is in one-to-one correspondence
		with $n$-hallways having lines of sight of slope $\alpha$.  Applying Proposition \ref{PropNumberOfPartitionElements}
		shows $|\mathcal Y_{\alpha, n}|\leq n+1$, which completes the proof.
	\end{proof}

	Recalling the correspondence between $n$-hallways and finite words,
	Corollary \ref{PropNPlus1} can be applied to show that rotation sequences satisfy the complexity
	conditions required of Sturmian sequences.
	We might also ask the total number of $n$-hallways admitting lines
	of slope of any $\alpha\in[0,1)$.

	\begin{theorem}[Mignosi \cite{mignosi}]
		Let $C(n)$ be the number of distinct $n$-hallways with $D_0=(0,1)$ and admitting
		a line of sight with slope in $[0,1)$.  Then,
			\[
				C(n) =1+ \sum_{i=1}^n (n+1-i)\phi(i)
			\]
		where $\phi(i)$ is Euler's totient function, which counts the number of integers in $\{1,\ldots, i\}$
		that are relatively prime to $i$.
	\end{theorem}

	In \cite{mignosi}, Mignosi uses combinatoric properties to count subwords of Sturmian sequences,
	In \cite{berstel}, Berstel and Pocchiola use geometric arguments to arrive at the same conclusion.

	Let's get a slightly better idea of what the set of all $n$-hallways looks like.

	\begin{definition}
		Let $S_n\subset [0,1]\times(0,1)$ be the set of pairs $(\alpha,\beta)$ such that
		$\ell_{\alpha\beta}$ is a line of sight for some $n$-hallway.
		Let $\mathcal P_n$ be the partition of $S_n$ where 
		$(\alpha,\beta)$ and $(\alpha',\beta')$ are in the same partition element
		if $\ell_{\alpha\beta}$ and $\ell_{\alpha'\beta'}$ are lines of sight for the same
		$n$-hallway.
	\end{definition}

	Corollary \ref{PropUniqueHallway} ensures that $\mathcal P_n$ is well defined.  Drawing $\mathcal P_n$
	as a subset of $\R^2$, we see that the vertical fiber of $\mathcal P_n$ with $x$-coordinate $\alpha$
	is precisely $\mathcal Y_{\alpha,n}$.

	Looking at $\mathcal Y_{\alpha,n}$ as a function of $\alpha$, we wee that $\mathcal Y_{\alpha,n}$ must be
	split at the point 
	\[
		f_i(\alpha) = \proj_\alpha(\{i\}\times \Z) = \proj_\alpha(i,0)\mod 1 = -\alpha i\mod 1
	\]
	for every $i\in\{0,\ldots, n\}$.  In other words, $\mathcal P_n$ looks like $[0,1]\times (0,1)$ cut 
	by the lines (mod 1) of slope $-i$ for $i\in\{0,\ldots, n\}$.  See Figure \ref{FigBasicpartition}.

		\begin{figure}[h!]
			\footnotesize
			\begin{center}

			\begin{tikzpicture}[scale=6]
			    \draw[thick] (0,0) grid (1,1);

			    \draw[] (0,1) -- (1,0);
			    
			    \draw[] (0,1) -- (1/2,0);
			    \draw[] (1/2,1) -- (1,0);
			    
			    \draw[] (0,1) -- (1/3,0);
			    \draw[] (1/3,1) -- (2/3,0);
			    \draw[] (2/3,1) -- (1,0);

			    \draw[] (0,1) -- (1/4,0);
			    \draw[] (1/4,1) -- (2/4,0);
			    \draw[] (2/4,1) -- (3/4,0);
			    \draw[] (3/4,1) -- (1,0);
			    
			    \draw[] (0,1) -- (1/5,0);
			    \draw[] (1/5,1) -- (2/5,0);
			    \draw[] (2/5,1) -- (3/5,0);
			    \draw[] (3/5,1) -- (4/5,0);
			    \draw[] (4/5,1) -- (1,0);

			    \draw[] (0,0) -- (1,0) node[right] {$\alpha$};
			    \draw[] (0,0) -- (0,1) node[above] {$t$};

			    \draw[] (1/2,.02) -- (1/2,0) node[below, yshift=-1] {$\frac{1}{2}$};
			    \draw[] (1/3,.02) -- (1/3,0) node[below, yshift=-1] {$\frac{1}{3}$};
			    \draw[] (2/3,.02) -- (2/3,0) node[below, yshift=-1] {$\frac{2}{3}$};
			    \draw[] (0,.02) -- (0,0) node[below, yshift=-1] {$0$};
			    \draw[] (1,.02) -- (1,0) node[below, yshift=-1] {$1$};
			    
			    \draw[] (.02,1) -- (0,1) node[left, xshift=-1] {$1$};
			    \draw[] (.02,1/2) -- (0,1/2) node[left, xshift=-1] {$\frac{1}{2}$};

			\end{tikzpicture}
			\end{center}
			\caption{ \footnotesize 
				The partition $\mathcal P_5$.
			}
			\label{FigBasicpartition}
		\end{figure}
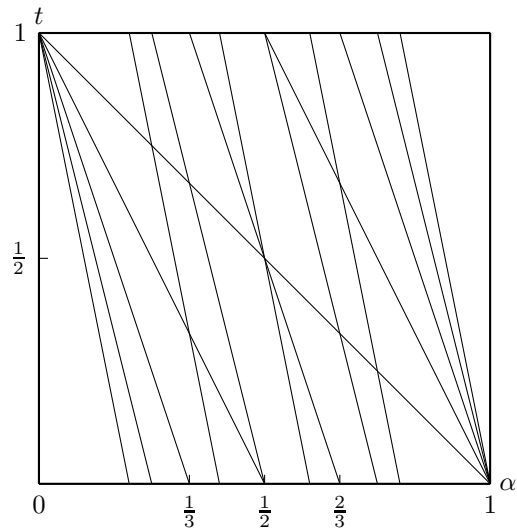

	Using this description of $\mathcal P_n$, we can answer our question about rational lines of sight.

	\begin{theorem}\label{PropBigTheorem}
		Given an $n$-hallway $H_n$ admitting a line of sight, there is a
		rational line of sight $\ell_{\alpha\beta}$ with $\alpha=\frac{p}{q}$ and
		$q\leq n$.
	\end{theorem}

	\begin{proof}
		Fix an $n$-hallway $H_n$ admitting a line of sight at let $P\in \mathcal P_n$ be
		the corresponding partition element.  That is, for every $(\alpha,\beta)\in P$, 
		$\ell_{\alpha\beta}$ is a line of sight for $H_n$.

		Our proof would be complete if we could show that $P$ contained a point $(\alpha,\beta)$
		where $\alpha=\frac{p}{q}$ with $q\leq n$.  To this end, consider the corners of $P$,
		when $P$ is interpreted as a polygon.  Since $\mathcal P_n$ is formed by cutting
		$[0,1]\times (0,1)$ by lines of slope $-1,\ldots,-n$, the edges of $P$ are segments
		of lines of the same slope and so the corners are intersections of such lines.

		We will now compute the intersection of two lines of the form $y=-ax\mod 1$.
		Notice that any connected segment of the graph of such a line is identical to the graph
		of the line $y=-ax+b$ restricted to $[0,1]\times(0,1)$ for some $1\leq b\leq a$.

		Let $L_{-a,b}(x) = -ax+b$ be a line with slope $-a$ and $y$-intercept $b$.  Then,
		the intersection of the graphs of $L_{-a,b}$ and $L_{-a',b'}$ occurs at
		\[
			x=\frac{b'-b}{a'-a}\qquad\text{and}\qquad L_{-a,b}(x)=\frac{a'b-ab'}{a'-a}.
		\]

		Now, if $a,a'\in\{0,\ldots, n\}$, $|a'-a|\leq n$.  This shows that the $x$-coordinate
		of every corner of $P$ is of the form $\frac{p}{q}$ with $q\leq n$.

		To complete the proof, notice that either $P$ is one of the two extreme cases---the
		triangles with corners $(0,0)$, $(1,1)$, $(1/n,0)$ or $(1,0)$, $(1,1)$, $(1/n,1)$---or
		$P$ has a corner directly above or below its interior (see Figure \ref{FigBasicpartition}).
		If $P$ is the left extremal triangle, then there is a line of sight $\ell_{0\beta}$
		for some $\beta$ and if $P$ is the right extremal triangle, there is a line of sight $\ell_{1\beta}$
		for some $\beta$.

		Finally, if $P$ has a corner with coordinates $(\alpha,b)$ above or below its interior,
		then there must be some point $(\alpha,\beta)\in P$.  Thus $\ell_{\alpha\beta}$ is a line
		of sight and as shown, $\alpha=\frac{p}{q}$ with $q\leq n$.

	\end{proof}

\section{Infinite Hallways}

	Diagrams like Figure \ref{FigBasicpartition} show that if an $n$-hallway admits
	a line of sight $\ell_{\alpha\beta}$, then it admits lines of sign $\ell_{\gamma\delta}$
	for a host of real numbers $\gamma$ and $\delta$.  However, things become a bit
	more restricted when we pass to infinite hallways.

	An infinite hallway is defined analogously to a finite hallway and can be thought of as
	the union of finite hallways that get longer and longer.  
	It would now seem natural to say that an infinite hallway $H^\infty$ has a line
	of sight if and only if $\ell_{\alpha\beta}\cap H^\infty=\emptyset$ for some $\alpha$ and $\beta$.  However, this definition
	rules out a very desirable property.

	\emph{Desirable Property}:
	\vspace{-1em}
	\begin{quote}
		If $H^\infty$ is an infinite hallway admitting a line of sight, then
		there exists a doorway $D_{-1}=(d_{-1},d_{-1}+1)$ such that the infinite
		hallway $\{-1\}\times (\R\backslash D_{-1})\cup H^{\infty}$ admits a line of sight.
	\end{quote}

	That is, if a hallway admits a line of sight, we should be able to add a doorway to it and
	have it still admit a line of sight.  In the finite hallway case, this is always true 
	and there are always infinitely many lines of sight.  However, for infinite hallways,
	there may be a unique line of sight and the na\"ive formulation of infinite hallways
	does not always allow a door to be added while preserving visibility.  For example, consider the
	infinite hallway $H^\infty$ whose $i$th doorway is $D_i=(\floor{n\pi},\floor{n\pi}+1)$ for $i\geq 1$
	and where $D_0$ is undefined.  As will be shown in Theorem \ref{PropInfiniteHallwaysUniqueSlope},
	$\ell_{\pi0}$ is the only line such that $\ell_{\pi0}\cap H^\infty=\emptyset$.  However, $\ell_{\pi0}$
	contains the point $(0,0)$ on the integer lattice and so there is no acceptable choice of $D_0$ if we
	would like $H^\infty$ to admit a line of sight (since every doorway excludes every lattice point).

	We will solve this issue by introducing infinitesimals.

	\begin{definition}[Infinitesimals]
		Let $\epsilon$ represent a positive infinitesimal.  Formally, let $\R^\epsilon = \{a+b\epsilon:a,b\in\R\}$
		be the two-dimensional vector space with basis $\{1,\epsilon\}$ over the field $\R$ endowed with the following
		total order:
		\[
			a+b\epsilon < c+d\epsilon \qquad\text{if}\qquad a<c\qquad \text{or}\qquad a=c\text{ and } b<d
		\]
		and 
		\[
			a+b\epsilon = c+d\epsilon \qquad\text{if}\qquad a=c\text{ and } b=d.
		\]
		
		A number $r=a+b\epsilon\in\R^\epsilon$ is called \emph{real} if $b=0$.
	\end{definition}

	This all amounts to saying that $\epsilon$ satisfies
	$0<\epsilon<a$ for all positive real numbers $a$ and that addition of infinitesimals and
	multiplication of infinitesimals by real numbers makes sense.  We define open and closed
	intervals in the usual way: 
	the open interval
	$(a,b)\subseteq \R^\epsilon$ is defined as $(a,b)=\{r\in\R^\epsilon:a<r<b\}$ and the closed
	interval $[a,b]\subset \R^\epsilon$ is defined as $[a,b]=\{r\in\R^\epsilon:a\leq r\leq b\}$
	and in general we endow $\R^\epsilon$ with the order topology.

	Now we will precisely define what it means to be an infinite hallway, systematically replacing
	$\R$ (in the finite case) with $\R^\epsilon$ (in the infinite case).

	\begin{definition}[Infinite Hallway]
		Let
		$D_i=(d_i,d_i+1)\subset \R^\epsilon$ for $d_i\in\Z$ be a sequence of open unit intervals and define
		the corresponding infinite hallway $H^\infty$ as
		\[
			H^\infty = \bigcup_{i\in\Z} \{i\}\times (\R^\epsilon\backslash D_i).
		\]
		We will notate the restriction of $H^\infty$ to the first $n$ hallways it contains by
		\[
			H_n^\infty = H^\infty\cap ([0,n]\times\R^\epsilon).
		\]
	\end{definition}

	\begin{definition}[Infinite Lines]
		For $\alpha,\beta\in\R$, define the \emph{infinite line} $\ell_{\alpha\beta}^\epsilon\subset \R\times \R^\epsilon$ by
		\[
			\ell_{\alpha\beta}^\epsilon = \{(x, \alpha (x+t\epsilon)+\beta):x,t\in\R\}.
		\]
	\end{definition}
	
	Given a subset $X\subset\R\times\R^\epsilon$, for any $t\in\R^\epsilon$, we define $X+t=\{(a,b+t):(a,b)\in X\}$.
	In this notation, we can alternatively define $\ell^\epsilon_{\alpha\beta} = \bigcup_{t\in\R} \ell_{\alpha\beta}+t\epsilon$.

	Infinite lines are ``fattened up'' lines.  As such, we need to slightly modify 
	how we define a line of sight.  Whereas, $\ell_{\alpha\beta}\cap H=\emptyset$ captures
	the idea ``no light is blocked,'' we would like to capture the idea ``not all light is blocked.''

	\begin{notation}
		Given a line or infinite line $\ell_{\alpha\beta}^*$ and a hallway or infinite hallway $H^*$, the
		\emph{visibility operator} $\vee$ is defined as
		\[
			\ell_{\alpha\beta}^*\vee H^* = \text{real part of }(\proj_{\alpha}\ell_{\alpha\beta}^*)\cap (\proj_{\alpha}H^*)^c.
		\]
	\end{notation}

	\begin{definition}[Infinite Line of Sight]
		The infinite line $\ell_{\alpha\beta}^\epsilon$ is a \emph{line of sight for the infinite hallway $H^\infty$}
		if $\ell_{\alpha\beta}^\epsilon\vee H^\infty\neq \emptyset$.
	\end{definition}

	Notice that infinite lines of sight are still defined by real parameters, they are just
	infinitesimally ``fattened up.''  
	To see this, consider the following example.  Let $H^\infty$ be the infinite hallway
	with doorways $D_i=(i,i+1)$.  $H^\infty$ has infinite lines of sight $\ell_{1\beta}^\epsilon$ for every
	$\beta\in[0,1]$.  In particular, $\ell_{10}^\epsilon$ and $\ell_{11}^\epsilon$ are infinite lines of sight, but we
	would not consider the real lines $\ell_{10}$ and $\ell_{11}$ to be lines of sight.

	From now on we will refer to infinite lines of sight simply
	as lines of sight and use the term \emph{real line of sight} if we need to draw a careful
	distinction between infinite and non-infinite lines of sight.

	\begin{proposition}\label{PropLittleNudge}
		If $H^\infty$ is an infinite hallway with line of sight $\ell_{\alpha\beta}^\epsilon$,
		then either $(\ell_{\alpha\beta}+\epsilon)\cap H^\infty\neq \emptyset$ or  
		$(\ell_{\alpha\beta}-\epsilon)\cap H^\infty\neq \emptyset$.
	\end{proposition}
	\begin{proof}
		Observe that if $\ell^\epsilon_{\alpha\beta}\vee H^\infty\neq \emptyset$,
		then for some $r\in\R$ we have $(\ell_{\alpha\beta}+r\epsilon)\cap H^\infty = \emptyset$.
		If $r\geq 0$, then we must have $(\ell_{\alpha\beta}+\epsilon)\cap H^\infty\neq \emptyset$
		and if $r\leq 0$ we must have $(\ell_{\alpha\beta}-\epsilon)\cap H^\infty\neq \emptyset$.
	\end{proof}

	Next we will show that the using infinite lines of sight gives us our desired property.
	In fact, it gives us something slightly stronger.
	\begin{proposition}\label{PropDesireableProperty}
		Suppose $\hat H^\infty$ is an infinite hallway missing its first door and admitting
		a line of sight $\ell_{\alpha\beta}^\epsilon$.  If $D_i$ for $i\geq 1$ are the doorways
		for $\hat H^\infty$, then there exists a doorway $D_0$ such that the infinite hallway
		$H^\infty$ with doorways $D_i$ for $i\geq 0$ admits the line of sight $\ell_{\alpha\beta}^\epsilon$.
	\end{proposition}
	\begin{proof}
		Suppose $D_i$ for $i\geq 1$ and $\hat H^\infty$ are as in the statement of the proposition.
		Given a definition for $D_0$, let $H^\infty$ be the infinite hallway with doorways $D_i$ for $i\geq 0$.
		Let $z=\floor{\beta}$.  From Proposition \ref{PropLittleNudge}
		we know either $(\ell_{\alpha\beta}+\epsilon)\cap \hat H^\infty=\emptyset$ or $(\ell_{\alpha\beta}-\epsilon)\cap \hat H^\infty=\emptyset$.

		Suppose $(\ell_{\alpha\beta}+\epsilon)\cap \hat H^\infty=\emptyset$.  In this case,
		define $D_0=(z,z+1)$.  We now have $z+1>\proj_{\alpha} (\ell_{\alpha\beta}+\epsilon) = \beta+\epsilon > z$,
		and so $(\ell_{\alpha\beta}+\epsilon)\cap H^\infty=\emptyset$.  It immediately follows
		that $\ell_{\alpha\beta}^\epsilon$ is a line of sight for $H^\infty$.
		
		If $(\ell_{\alpha\beta}-\epsilon)\cap \hat H^\infty=\emptyset$, a similar argument shows
		that if $D_0=(z-1,z)$, then $H^\infty$ admits the line of sight $\ell^\epsilon_{\alpha\beta}$.
	\end{proof}

	\begin{theorem}\label{PropInfiniteHallwaysUniqueSlope}
		Let $H^\infty$ be an infinite hallway.  If $\ell_{\alpha\beta}^\epsilon$ and
		$\ell_{\gamma\delta}^\epsilon$ are both lines of sight for $H^\infty$, then $\alpha=\gamma$.
		In other words, all lines of sight for $H^\infty$ have the same slope.
	\end{theorem}
	\begin{proof}
		Let $H^\infty$ be an infinite hallway with doorways $D_i=(d_i,d_i+1)$
		and suppose $\ell_{\alpha\beta}^\epsilon$ is a line of sight for $H^\infty_n$.
		Let $\bar D_i=[d_i,d_i+1]$ and notice that 
		$\ell_{\alpha\beta}$ must pass through $\bar D_0$ and $\bar D_n$.
		From this, we conclude
		\[
			\frac{d_0-(d_n+1)}{n}=\frac{d_0-d_n}{n}-\frac{1}{n}\leq \alpha \leq \frac{d_0-d_n}{n}+\frac{1}{n}=\frac{d_0+1 - d_n}{n},
		\]
		since a line with any slope outside of that range would not pass through
		$\bar D_0$ and $\bar D_n$.

		Now, if $\ell_{\alpha\beta}^\epsilon$ is a line of sight for $H^\infty$, it is a line of sight for 
		every $H^\infty_n$.  Thus for any $r>0$ and $n>1/r$, 
		\[
			\alpha-r\leq \frac{d_0-d_n}{n}\leq \alpha+r
		\]
		and so $\lim_{n\to\infty}(d_0-d_n)/n=\alpha$ exists.  Since $\lim_{n\to\infty}(d_0-d_n)/n$ exists and is unique,
		there can only be one slope for lines of sight for $H^\infty$.
	\end{proof}

	\begin{definition}[Periodic Hallways]
		An infinite hallway $H^\infty$ with doors $D_i=(d_i,d_i+1)$ is
		called \emph{periodic with period $m>0$} if there exists some $k$
		such that
		\[
			D_{i+m} = D_i+k=(d_i+k,d_i+k+1).
		\]
		If there exists an $m$ such that $H^\infty$ is periodic with period $m$, $H^\infty$ is called
		\emph{periodic}, otherwise $H^\infty$ is called \emph{aperiodic}.

		If $H^\infty$ is periodic, the \emph{minimum period} is the smallest $m$
		such that $H^\infty$ is periodic with period $m$.
	\end{definition}

	\begin{theorem}
		\label{PropRatiffRatSight}
		An infinite hallway $H^\infty$ which admits a line of sight
		is periodic if and only if it admits a rational line of sight.
	\end{theorem}
	\begin{proof}
		Suppose $H^\infty$ is an aperiodic infinite hallway with doorways $D_i=(d_i,d_i+1)$
		and line of sight $\ell_{\alpha\beta}^\epsilon$.

		If $H^\infty$ is periodic, then there is some $m$ such that $d_{km} = d_0+k(d_m-d_0)$.
		Now, 
		\[
			\alpha=\lim_{n\to\infty}\frac{d_0-d_n}{n} = \lim_{k\to\infty}\frac{d_0-d_{km}}{km}
			=\lim_{k\to\infty}\frac{d_0-(d_0+k(d_{m}-d_0))}{km} = \frac{r}{m}
		\]
		where $r=d_m-d_0$ and so $\alpha\in \Q$.

		Conversely, suppose $\alpha=\tfrac{p}{q}\in\Q$. Now, $\beta+n\tfrac{p}{q}+\epsilon\in(d_n,d_n+1)$
		or $\beta+n\tfrac{p}{q}-\epsilon\in(d_n,d_n+1)$,
		and so in particular, $d_n$ is the unique integer such that
		\[
			\beta+n\tfrac{p}{q} -1 < d_n < \beta+n\tfrac{p}{q}.
		\]
		Letting $k=n+q$, we additionally have
		\[
			\beta+n\tfrac{p}{q} + p -1 < d_k=d_{n+q} < \beta+n\tfrac{p}{q} + p,
		\]
		and so $d_{n+q}$ is the unique integer such that $
		\beta+n\tfrac{p}{q} -1 < d_{n+q}-p < \beta+n\tfrac{p}{q}$, which shows $H^\infty$ is periodic
			with period $q$.
	\end{proof}
		
	The proof of Theorem \ref{PropRatiffRatSight} actually gives us an additional result.
	\begin{theorem}
		An infinite hallway $H^\infty$ admitting a line of sight has period $m$
		if and only if it admits a rational line of sight with slope $p/m$ for some $p\in\Z$.
	\end{theorem}
	
	\begin{corollary}
		Suppose $H^\infty$ is an infinite periodic hallway with minimal period $m$.
		If $\ell_{\tfrac{p}{m}\beta}^\epsilon$ is a line of sight 
		for $H^\infty$, then $p/m$ is in lowest terms.
	\end{corollary}
	\begin{proof}
		We will prove the contrapositive.
		Suppose $H^\infty$ is an infinite hallway admitting a line of sight $\ell_{\tfrac{p}{m}\beta}^\epsilon$.
		If $p/m$ is not in lowest terms, then $p/m=p'/m'$ where $0<m'<m$.  Since $m'$ is a period for $H^\infty$,
		we know $m$ is not a minimal period.
	\end{proof}

	We now know quite a bit about the slopes of lines of sight for infinite hallways.
	Let us introduce a lemma dealing with the intercepts.

	\begin{lemma}
		\label{PropDinterval}
		Suppose $H^\infty$ is an infinite hallway with doorways $D_i$ and admitting
		a line of sight $\ell_{\alpha\beta}^\epsilon$.  Let $P_n=(p^n_i)_{i\in\Z}$ be an
		enumeration of the points in $\proj_\alpha(\{0,1,\ldots,n\}\times\Z)$ satisfying
		$p^n_i < p^n_{i+1}$.  Let $B_n = \{ [p^n_i,p^n_{i+1}]: p^n_i\in P^n\}$ and let 
		$D = \{\gamma:\ell_{\alpha\gamma}^\epsilon\text{ is a line of sight for }H^\infty\}$.
		Then $D$ is a (possibly degenerate) interval such that for all $n$, $D\subseteq B$
		for some $B\in B_n$.
	\end{lemma}
	\begin{proof}
		Suppose $\ell_{\alpha\beta}^\epsilon$ is a line of sight for an infinite hallway $H^\infty$
		with doorways $D_i=(d_i,d_i+1)$.  Let $\bar D_i=[d_i,d_{i+1}]$.
		If $\ell_{\alpha\gamma}^\epsilon$ is a line of sight for $H^\infty$, then $\ell_{\alpha\gamma}^\epsilon$
		must intersect $\{i\}\times D_i$ for every $i$.  By Proposition \ref{PropLittleNudge},
		$\ell_{\alpha\gamma}+\epsilon$ or $\ell_{\alpha\gamma}-\epsilon$ must
		intersect $\{i\}\times D_i$ for each $i$ and so $\ell_{\alpha\gamma}$ intersects
		$\{i\}\times \bar D_i$ for each $i$.  Thus, we have the equality
		\[
			D = \bigcap_{i\geq 0} \proj_{\alpha}\left(\{i\}\times \bar D_i\right).
		\]
		Since $D$ can be written as an intersection of intervals, it is an interval.
		Lastly, since $D_i\cap \Z=\emptyset$ for all $i$, we know that for every $n$,
		$D\subset B$ for some  $B\in B_n$.
	\end{proof}

	From Lemma \ref{PropDinterval} we can get a bound on the size of $D$.

	\begin{theorem}
		Let $H^\infty$ be an infinite hallway.
		If $H^\infty$ is periodic with minimal period $m$ and $\ell_{\alpha\beta}^\epsilon$
		and $\ell_{\alpha\delta}^\epsilon$ are lines of sight for $H^\infty$, then
		$|\beta-\delta|\leq 1/m$.
	\end{theorem}
	\begin{proof}
		This follows quickly from Lemma \ref{PropDinterval}.  If $\ell^\epsilon_{\alpha\beta}$ is a
		line of sight for an infinite periodic hallway with minimal period $m$, then $\alpha=p/m$ for some $p$.
		This means, for all $i,j\geq m$, 
		\[
			\proj_\alpha(\{0,1,\ldots,i\}\times\Z) = \proj_\alpha(\{0,1,\ldots,j\}\times\Z),
		\]
		and so $B_i=B_j$ where $B_n$ is defined as in the statement of Lemma \ref{PropDinterval}.
		Since $\alpha=p/m$ must be in lowest terms, a quick calculation shows every interval in $B_i$ for $i\geq m$
		has width no greater than $1/m$, which completes the proof.
	\end{proof}

	\begin{theorem}
		\label{PropOneLineofSight}
		Let $H^\infty$ be an infinite hallway.
		If $H^\infty$ is aperiodic, then there is at most one line of sight for $H^\infty$.
	\end{theorem}
	\begin{proof}
		Suppose $H^\infty$ is an aperiodic infinite hallway with doorways $D_i=(d_i,d_i+1)$
		and line of sight $\ell_{\alpha\beta}^\epsilon$.  Necessarily we have $\alpha\notin \Q$, and by
		Theorem \ref{PropInfiniteHallwaysUniqueSlope} $\alpha$ is unique.

		Let $D=\{\gamma:\ell_{\alpha\gamma}^\epsilon\text{ is a line of sight for }H^{\infty}\}$.
		Since $\beta\in D$, $D$ is a non-empty (but possibly degenerate) interval.  Further, for every $n$
		we have that $D\subseteq B$ for some $B\in B_n$ where $B_n$ is as in the statement of Lemma \ref{PropDinterval}.
		But, since $\alpha\notin \Q$, $\proj_\alpha(\N\times\Z)$ is dense the diameter if
		every interval in $B_n$ tends towards zero as $n\to\infty$.  We conclude that $D=\{\beta\}$
		must be a singleton and so there is only one line of sight for $H^\infty$.
	\end{proof}
	
	The converse to Theorem \ref{PropOneLineofSight} is also true.  If there is a unique line of sight
	for an infinite hallway, it must be aperiodic.

\subsection{Metrics on Hallways}
	
	The theorems in the preceding section, taken together, show a correspondence between
	infinite hallways and a symbolic \emph{shift space}.  Consider
	the map
	$\Phi:\{\text{infinite hallways}\}\to\Z^\N$ where the $i$th coordinate of $\Phi(H)$ is $d_i-d_{i+1}$.
	$\Phi$ is an extension of the identically-named map between $n$-hallways and $n$-words from earlier.
	Under the assumption that the initial doorway of every infinite hallway is $D_0=(0,1)$, $\Phi$ is
	a bijection between hallways and sequences.  

	Let $T:\Z^\N\to\Z^\N$ be the \emph{shift} map.  That is $T(a_0,a_1,a_2,\ldots)=(a_1,a_2,\ldots)$
	deletes the first coordinate of a sequence.
	Let $\Omega$ be the image under $\Phi$ of all infinite hallways admitting a line of sight.  
	Now, $T(\Omega)\subset \Omega$ is immediate, and Proposition \ref{PropDesireableProperty} (the
	proposition that gives us our desirable property) shows that $T(\Omega)=\Omega$.  Thus, 
	$\Omega$ is $T$-invariant.  The word closed is almost always used in conjunction with the word
	invariant, so we might ask if $\Omega$ is also closed.

	The shift space $(\Z^\N, T)$ is typically endowed with the product topology on $\Z^\N$
	where $\Z$ has the discrete topology.  This is the same topology arising from
	the \emph{standard metric on sequences}, $d$.  Namely, if $x,y\in \Z^\N$, $d(x,y)=1/n$
	where $n$ is the index of the first coordinate where $x$ and $y$ differ.

	Using $\Phi$, the standard metric on sequences induces a metric on infinite hallways.

	\begin{definition}[Standard Metric on Infinite Hallways]
		Let $H,H'$ be infinite hallways.  The \emph{standard metric} on infinite
		hallways, notated $d_S$, is defined as
		\[
			d(\Phi(H),\Phi(H'))=d_S(H,H') = \frac{1}{n}\quad \text{ where }\quad
			n=\inf \{k\in\N: H_k \neq H_k'\},
		\]
		with the convention $1/\infty=0$ and $1/0=\infty$.
	\end{definition}

	Standard arguments now show that the set of all infinite hallways is complete with respect
	to $d_S$.

	Let $V:\{\text{infinite hallways}\}\to\{0,1\}$ be the \emph{visibility function}.  That is, $V(H)=1$ if $H$
	admits a line of sight and $0$ otherwise.  Now, suppose $H$ is an infinite hallway that admits
	a line of sight and let $H'$ be $H$ with a single doorway changed.  Since $V(H)=1$
	and $V(H')=0$, we cannot hope that $V$ is continuous.  However, we might hope that $V$ would be upper-semicontinuous.  
	That is, we might hope that if $H^{(n)}\to H$ is a convergent sequence of infinite hallways and $V(H^{(n)})=1$,
	then $V(H)=1$.
	Alas, this is not so with the standard metric.

	\begin{proposition}
		\label{PropDiscont}
		The visibility function $V$ is not upper-semicontinuous with respect to the standard metric on
		infinite hallways, $d_S$.
	\end{proposition}
	\begin{proof}
		Let $H^{(n)}$ be the periodic infinite hallway admitting a line of sight with slope
		$1/n$ and with doorways $D_0^{(n)}=(0,1)$, $D_1^{(n)}=(1,2)$, and $D_i^{(n)}=(\floor{\tfrac{i-1}{n}}+1, \floor{\tfrac{i-1}{n}}+2)$.
		Geometrically, $H^{(n)}$ has a jump of size 1 between $D_0$ and $D_1$, then has $n$
		identical doorways in a row before another jump of size 1.

		Now, for all $m,n$, we have
		\[
			(0,1)=D_0^{(m)}=D_0^{(n)}\qquad\text{and}\qquad(1,2)= D_1^{(m)}=D_1^{(n)}
		\]
		and for every $n,m>k>1$ and $1\leq i\leq k$,
		\[
			(1,2) = D_1^{(n)} = D_i^{(m)}.
		\]
		 From this description, we see $H^{(n)}\to H^\infty$
		which has doorways $D_0^\infty =(0,1)$ and $D_i^\infty=(1,2)$ for all $i\geq 1$.  But $V(H^{(n)})=1$
		and $V(H^\infty)=0$, so $V$ is not semi-continuous.
	\end{proof}

	Since $\Omega=\Phi^{-1}\circ V^{-1}(1)$, Proposition \ref{PropDiscont} shows that $\Omega$ is
	not closed with respect to the standard metric.  Similarly, the set of all Sturmian sequences
	is not closed under the standard metric (because, as in Proposition \ref{PropDiscont},
	limits of periodic points may be aperiodic
	but eventually periodic) and the property of being a rotation sequence is not closed under limits.

	All hope is not lost, though.  There may be a different metric that $V$ is upper-semicontinuous with
	respect to.  The counterexample used in Proposition \ref{PropDiscont} relied on a sequence
	of periodic infinite hallways.  In particular, the lines of sight had slope converging to a
	rational number, so we might seek to prevent hallways from
	doing this. 

	\begin{definition}[Common Initial Segment]
		Given two infinite hallways $H$ and $H'$, their \emph{common
		initial segment} is 
		\[
			\comm(H,H') = H_{1/d_S(H,H')} = H_{1/d_S(H,H')}'.
		\]
	\end{definition}
	If $H=H'$, we consider $\comm(H,H')=H=H'$.  Otherwise, $\comm(H,H')$ is always a finite hallway or empty.

	\begin{definition}[Unframed Hallway]
		Given a hallway $H$ with $i$th doorway $(d_i,d_i+1)$, the corresponding
		\emph{unframed hallway} is the hallway $\bar H$ whose $i$th doorway
		is the closed interval $[d_i,d_i+1]$.
	\end{definition}

	\begin{definition}[Rational Metric on Infinite Hallways]
		Let $H,H'$ be infinite hallways.  The \emph{rational metric} on infinite hallways,
		notated $d_R$, is defined as follows.

		If $H=H'$, then $d_R(H,H')=0$; if $\comm(H,H')$ admits a line of sight,
		\[
			d_R(H,H') = \max\{\tfrac{1}{q}: \ell_{\frac{p}{q}\beta}\text{
			is a line of sight for } \comm(\bar H,\bar H')\text{ for some }\beta\in\R, p\in \Z\};
		\]
		and if $\comm(H,H')$ admits no line of sight and $H\neq H'$, then $d_R(H,H')=\infty$.
	\end{definition}

	\begin{proposition}
		The rational metric, $d_R$, is a metric on infinite hallways.
	\end{proposition}
	\begin{proof}
		By definition $d_R(H,H')=0$ if $H=H'$.  Suppose $H\neq H'$ .  If
		$\comm(H,H')$ admits no line of sight, $d_R(H,H')=\infty>0$.
		If $\comm(H,H')$ admits a line
		of sight, because it is a finite hallway, it admits a rational line of
		sight and so $d_R(H,H')>0$.  Further, since $\comm(H,H')=
		\comm(H',H)$, $d_R(H,H')=d_R(H',H)$ and so conditions (i) and (ii) are
		satisfied.

		Now we consider condition (iii).  Let $H,H',H''$ be infinite hallways and notice
		that either
		\[
			\comm(H,H'')\subseteq \comm(H,H')\qquad\text{or}\qquad\comm(H',H'')\subseteq \comm(H,H').
		\]
		To see this, let $n_{XY}=1/d_S(X,Y)$ be the number of doorways that hallways $X$
		and $Y$ agree for, and consider the three choices: (a) $n_{HH'}=n_{HH''}$, (b) 
		$n_{HH'} > n_{HH''}$, or (c) $n_{HH'}<n_{HH''}$.

		In case (a), $\comm(H,H'')= \comm(H,H')$ and so $\comm(H,H'')\subseteq \comm(H,H')$;
		in case (b), we must have $n_{H'H''} = n_{HH''} < n_{HH'}$ and so $\comm(H',H'')\subseteq \comm(H,H')$; and,
		in case (c), we must have $n_{H'H''} = n_{HH'}$ which means $\comm(H',H'')= \comm(H,H')$ and so
		$\comm(H',H'')\subseteq \comm(H,H')$.
		
		Now, for finite hallways $X,Y$ where $X\subset Y$, the set of lines of sight for $X$ is a superset of
		the set of lines of sight for $Y$.  Thus, the above set inclusions give us either
		\[
			d_R(H,H'') \geq d_R(H,H')\qquad\text{or}\qquad d_R(H',H'') \geq d_R(H,H'),
		\]
		and so certainly $d_R(H,H'')+d_R(H',H'')\geq d_R(H,H')$.
	\end{proof}

	\begin{notation}
		Let $\mathcal H$ be the set of all infinite hallways; 
		let $\mathcal Q$ be the set of all infinite periodic hallways that admit lines
		of sight; let $\bar {\mathcal Q}\subset \mathcal H$ be the closure of $\mathcal Q$ under the 
		metric $d_R$; and let $\mathcal Q^c=\bar{\mathcal Q}\backslash \mathcal Q$.
	\end{notation}

	It is not immediately obvious that $\mathcal Q^c$ contains anything at all, but we will show
	that $\mathcal Q^c$ is precisely the set of all aperiodic hallways admitting lines of sight.
	Not only that, but we will show that $\mathcal Q^c$ is a closed set with respect to $d_R$, from
	which it will follow that $V$, the visibility function, is upper-semicontinuous.

	\begin{proposition}\label{PropQcClosed}
		The set $\mathcal Q^c$ is closed with respect to $d_R$.
	\end{proposition}
	\begin{proof}
		By definition $\bar {\mathcal Q}$ is closed with respect to $d_R$.  Thus, if we
		can show $\mathcal Q$ is open with respect to $d_R$, the proof will be complete.

		To that end, suppose $H\in \mathcal Q$ admits a line of sight $\ell_{\frac{p}{q}\beta}^\epsilon$.
		We necessarily have that $\ell_{\frac{p}{q}\beta}^\epsilon$ is a line of sight
		for $H_k$ for all $k$.  Thus, $d_R(H,H')\geq 1/q$ for all $H'\neq H$, and so the
		singleton $\{H\}$ is open with respect to $d_R$ (it is the only
		element in the $d_R$-ball of radius $1/(2q)$ with center $H$).  We conclude that 
		$\mathcal Q$ is the union of open sets and is therefore open.
	\end{proof}

	\begin{lemma}\label{PropSightForUnframedStillSame}
		If $H$ is an infinite hallway that admits a unique line of sight $\ell_{\alpha\beta}^\epsilon$, then
		the unframed infinite hallway $\bar H$ admits $\ell_{\alpha\beta}^\epsilon$ as its unique line
		of sight.
	\end{lemma}
	\begin{proof}
		Let $H$ be an infinite hallway with unique line of sight $\ell_{\alpha\beta}^\epsilon$
		and $\bar H$ the corresponding unframed infinite hallway.
		Notice that the proof of Theorem \ref{PropInfiniteHallwaysUniqueSlope} applies equally
		well to unframed infinite hallways, and so the slope of any line of sight for $\bar H$
		must be $\alpha$.

		Let $(d_i,d_i+1)$ be the $i$th doorway for $H$.  Similarly to the proof
		of Lemma \ref{PropDinterval}, 
		\[
			\{\beta\}=B=\bigcap_{n\in \N}(\proj_\alpha(\{n\}\times [d_i,d_i+1]),
		\]
		is the complete set of intercepts for the infinite hallway $H$ and is also
		the complete set of intercepts for the unframed infinite hallway $\bar H$.  Thus $\ell_{\alpha\beta}^\epsilon$
		is unique.
	\end{proof}

	\begin{proposition}\label{PropQcContainsAperiodics}
		Suppose $H$ is an aperiodic infinite hallway admitting a line of sight.  Then $H\in\mathcal Q^c$.
	\end{proposition}
	\begin{proof}
		Since $H$ is an aperiodic infinite hallway, $H\notin \mathcal Q$.  To show $H\in\mathcal Q^c$,
		we must show that there is a sequence $H^{(i)}$ of periodic infinite hallways such that
		$H^{(i)}\to H$ with respect to $d_R$.

		Let $\ell_{\alpha\beta}^\epsilon$ be the unique line of sight for $H$
		and let $(d_i,d_i+1)$ be the $i$th doorway of $H$.  Without loss of generality,
		assume $\alpha\in[0,1]$.  By Theorem \ref{PropRatiffRatSight},
		$\alpha\notin\Q$.
		By Lemma \ref{PropSightForUnframedStillSame}, the unframed infinite hallway $\bar H$
		admits the unique line of sight $\ell_{\alpha\beta}^\epsilon$.  

		Recall that $H_n$ is the restriction of $H$ to the first $n$ doorways.  Fix $q\in\N$.  Now,
		since there are only a finite number of rationals of the form $\frac{p}{q}\in[0,1]$,
		it must be the case that for large enough $n$, $H_n$ admits no line of sight of the form
		$\ell_{\tfrac{p}{q}\delta}$, lest $\alpha=\tfrac{p}{q}$.  Similarly, for large enough
		$n$, $\bar H_n$ admits no line of sight of the form $\ell_{\tfrac{p}{q}\delta}$.

		To complete the proof, notice that since $H_n$ is a finite hallway that admits a line of
		sight, it admits a rational line of sight $\ell_{\gamma_i\delta_i}$.  Thus, there exists a 
		periodic infinite hallway $H^{(n)}$ admitting the line of sight $\ell_{\gamma_i\delta_i}$ and satisfying
		\[
			H_n \subseteq \comm(H^{(n)},H).
		\]
		Further, by construction $d_R(H^{(n)},H)\to 0$.
	\end{proof}

	An immediate corollary of Proposition \ref{PropQcContainsAperiodics} is that $\mathcal Q^c$ is non-empty.
	Next, we will show that every hallway in $\mathcal Q^c$ admits a line of sight.

	\begin{lemma}\label{PropLinesOfSightClosedinUnframed}
		Let $H$ be an infinte hallway and $H_n$ be the restriction of $H$ to the first $n$ doorways.
		If $H_n$ admits a line of sight for every $n$, then the unframed hallway $\bar H$ admits a line
		of sight.
	\end{lemma}
	\begin{proof}
		For a hallway $K$, let 
		\[
			A(K) = \{\alpha: \ell_{\alpha\beta}^\epsilon\text{ is a line of sight for }K\text{ for some }\beta\}.
		\]
		By assumption, $A(H_n)\neq \emptyset$, and $A(H_m)\subseteq A(H_n)$ for all $m>n$.  Further notice
		that $
			A(H) = \bigcap_{n\in\N} A(H_n).
		$
		Let $\bar H_n$ be the unframed hallway corresponding to $H_n$.  Since $A(\bar H_n)$ is closed and bounded
		and $\{A(\bar H_n)\}$ satisfies the finite intersection property, $A(\bar H) = \bigcap_{n\in\N} A(\bar H_n)\neq \emptyset$,
		and so $\bar H$ admits a line of sight.
	\end{proof}

	\begin{proposition}\label{PropQcAdmitsLinesofSight}
		If $H\in\mathcal Q^c$, then $H$ admits a line of sight.
	\end{proposition}
	\begin{proof}
		Let $H\in\mathcal Q^c$ and suppose the sequence $H^{(i)}\in\mathcal Q$ satisfies $H^{(i)}\to H$ with
		respect to $d_R$. Let $D_i=(d_i,d_i+1)$ be the $i$th doorway of $H$.

		Let $|\!\comm(H^{(n)},H)|$ denote the number of doorways in $\comm(H^{(n)},H)$.
		Since $\comm(H^{(n)},H)$ admits a line of sight and is necessarily a finite
		hallway, $\comm(H^{(n)},H)$ must admit a rational line of sight with slope $p/q$
		where $q\leq |\!\comm(H^{(n)},H)|$.  Since $d_R(H^{(i)},H)\to 0$, we must have
		that $|\!\comm(H^{(n)},H)|\to\infty$.

		Now, by Lemma \ref{PropLinesOfSightClosedinUnframed}, the unframed hallway $\bar H$
		must admit a line of sight $\ell_{\alpha\beta}^\epsilon$ since $H_k\subseteq \comm(H^{(n)},H)$ for large enough $n$.
		Further, $\alpha\notin \Q$ and so $\ell_{\alpha\beta}^\epsilon$ must be unique.

		If we can show that $\ell_{\alpha\beta}^\epsilon$ is a line of sight for $H_k$, regardless of $k$,
		then $\ell_{\alpha\beta}^\epsilon$ will be a line of sight for $H$.  Suppose this is not the case,
		and let $k$ be the smallest number such that $\ell_{\alpha\beta}^\epsilon$ is not a line of sight for
		$H_k$. Trivially, $k\geq 1$, and since $H^{(n)}\to H$ with respect to $d_R$ and $H_k\subseteq \comm(H^{(n)},H)$ for large enough $n$,
		we must have that $\ell_{\alpha\beta}^\epsilon$ is a line of sight for $\bar H_k$.  

		Since $\ell_{\alpha\beta}^\epsilon$
		is a line of sight for $H_{k-1}$ but not for $H_k$, we must have either $\alpha k+\beta=d_k$
		or $\alpha k+\beta=d_{k}+1$.

		Assume $\alpha k+\beta=d_k$.  Since $\ell_{\alpha\beta}^\epsilon$ is not a line of sight
		for $H_k$, we have that $\ell_{\alpha\beta}+\epsilon$ is not a line of sight for $H_k$.
		Since $\alpha k+\beta+\epsilon\in D_k$, this means $\ell_{\alpha\beta}+\epsilon$ is not a line
		of sight for $H_{k-1}$.  So, by Proposition \ref{PropLittleNudge}, $\ell_{\alpha\beta}-\epsilon$
		must be a line of sight for $H_{k-1}$.  We conclude that for some $0\leq i<k$, we must have
		$\alpha i+\beta = d_i+1$.  Thus, $\alpha=\frac{d_k-d_i-1}{k-i}\in\Q$, which is a contradiction.

		Assuming $\alpha k+\beta =d_k+1$, the proof follows similarly.
	\end{proof}

	\begin{corollary}\label{PropQbarAdmitsLinesofSight}
		If $H\in \bar {\mathcal Q}$ then $H$ admits a line of sight.
	\end{corollary}
	\begin{proof}
		By definition $\bar{\mathcal Q}=\mathcal Q\cup \mathcal Q^c$.
		If $H\in\mathcal Q$, then by definition it admits a line of sight and by Proposition \ref{PropQcAdmitsLinesofSight},
		if $H\in\mathcal Q^c$, $H$ admits a line of sight.
	\end{proof}

	We are almost ready to prove the semi-continuity of $V$.  But first, let us completely characterize
	the set of hallways that admit lines of sight.

	\begin{theorem}\label{PropVisIsQbar}
		Let $\mathcal V= \{H\in\mathcal H:V(H)=1\}$ be the set of infinite hallways that admit lines of sight.
		Then, $\bar {\mathcal Q} =\mathcal V$.
	\end{theorem}

	\begin{proof}
		By Corollary \ref{PropQbarAdmitsLinesofSight}, $\bar{\mathcal Q}\subseteq \mathcal V$.  Now,
		suppose $H\in \mathcal V$.  If $H$ is a periodic hallway, then $H\in\mathcal Q$.

		Suppose $H$ is an aperiodic hallway.  Since $H$ admits a line of sight, so 
		does the finite hallway $H_k$.  Thus $H_k$ admits a rational line of sight $\ell_{\alpha_k\beta_k}$
		where $\alpha_k=\frac{p_k}{q_k}$
		by Theorem \ref{PropRatiffRatSight}.  Let $H^{(k)}$ be the infinite periodic hallway admitting the line of sight
		$\ell_{\alpha_k\beta_k}^\epsilon$.  We will now show $H^{(k)}\to H$ with respect to $d_R$.

		Since $H^{(k)}_k= H_k$, we have $H^{(k)}\to H$ with respect to $d_S$.  Let $\ell_{\frac{p_k}{q_k}\delta_n}$
		be a line of sight for $\bar H_k$ such that $q_k$ is as small as possible.  If $q_k\to \infty$, we are done.
		Suppose $\{q_k\}$ is bounded and let $\bar D_0=[d_0,d_0+1]$ and $\bar D_1=[d_1,d_1+1]$ be the first two doorways of $\bar H$.
		Since $d_1-d_0-1\leq \frac{p_k}{q_k} \leq d_1+1-d_0$ and there are only finitely
		many $q_k$, there exists $p/q\in\{p_k/q_k:k\in\N\}$ such that $p/q=p_k/q_k$ for infinitely many $k$.    
		It follows that $\proj_{p/q} \bar H_k$ is a non-empty closed interval for all $k$ and so $\proj_{p/q} \bar H$ is non-empty.
		Thus, $\bar H$ admits a rational line of sight with slope $p/q$, which is a contradiction.
	\end{proof}

	\begin{theorem}
		$V:\mathcal H\to\{0,1\}$ is upper-semicontinuous with respect to $d_R$.
	\end{theorem}
	\begin{proof}
		Let $H^{(n)}\to H$ with respect to $d_R$ and suppose $V(H^{(n)})=1$.  By Theorem \ref{PropVisIsQbar},
		$H^{(n)}\in \bar{\mathcal Q}$, and so by definition $H\in\bar{\mathcal Q}$.  Now, by Corollary
		\ref{PropQbarAdmitsLinesofSight}, $V(H)=1$.
	\end{proof}

\section{Applications of $d_R$}

	Theorem \ref{PropInfiniteHallwaysUniqueSlope} states that if an infinite
	hallway admits a line of sight, its slope is unique.  Thus, we may define a function
	$s:\bar {\mathcal Q}\to\R$ by
	\[
		s(H) = \alpha\text{ where }\ell_{\alpha\beta}^\epsilon\text{ is a line of sight for }H.
	\]
	Let $s^{-1}$ be the set-valued right-inverse to $s$.  That is
	\[
		s^{-1}(\alpha) = \{H:s(H)=\alpha\},
	\]
	and we have the equality $s\circ s^{-1} = \mathrm{id}$. Now, any metric $d$ on 
	infinite hallways induces a metric $\tilde d$ on $\R$ via
	\[
		\tilde d(\alpha, \gamma) = d(s^{-1}(\alpha),s^{-1}(\gamma)).
	\]

	Two metrics $d_X$ and $d_Y$ are said to be \emph{equivalent} if their convergent sequences
	are the same.  That is, $d_X(x_i, x)\to 0$
	if and only if $d_Y(x_i,x)\to 0$ for all sequences $(x_i)$.
	\begin{proposition}
		The metric $\tilde d_S$ on $\R$ induced by the metric $d_S$ is 
		equivalent to the standard metric  $d$ on $\R$ given by $d(\alpha,\gamma) = |\alpha-\gamma|$.
	\end{proposition}
	\begin{proof}
		We will first show that convergence in $d$ implies convergence in $\tilde d_S$.
		Let $(x_i)$ be a sequence and suppose $d(x_i,x)=|x_i-x|\to 0$.  Fix $k>0$ and let 
		$H_k$ be a finite hallway with doorways $D_i$ admitting a line of sight $\ell_{x\beta}$
		for some $\beta$.  Now, for any $y\in\R$,
		define
		\[
			D^y = \bigcap_{i\leq k} \proj_y(\{i\}\times D_i)
		\]
		and note that $D^y=(l^y,r^y)$ is always an open interval or the empty set.  
		Further, $D^{x_i}\to D^x$
		in the sense that $l^{x_i}\to l^x$ and $r^{x_i}\to r^x$.  Necessarily we have
		$\beta\in D^x$, but we also see that since $D^{x_i}\to D^x$, for all large enough $i$,
		we have
		$\beta\in D^{x_i}$.  Thus, for large enough $i$, $\ell_{x_i\beta}$ and $\ell_{x\beta}$
		are lines of sight for $H_k$ and so $\tilde d_S(x_i,x)\leq 1/k$.  But $k$
		was arbitrary, so $\tilde d_S(x_i,x)\to 0$.

		Now, suppose $\tilde d_S(x_i,x)\to 0$.  Fix $k>0$.  Now for all sufficiently
		large $i$, $\tilde d_S(x_i,x)\leq 1/k$.  Supposing $i$ is sufficiently
		large, we necessarily
		have that for some $\beta,\beta_i$, there exists a $k$-hallway, $H_k$, for which
		$\ell_{x\beta}$ and $\ell_{x_i\beta_i}$ are both lines of sight.  In particular,
		$\ell_{x\beta}$ and $\ell_{x_i\beta_i}$ both pass through the $k$th doorway of
		$H_k$ and so
		\[
			|(kx+\beta)-(kx_i+\beta_i)| \leq 1.
		\]
		By the reverse triangle inequality we have
		\[
			k|x-x_i|-|\beta-\beta_i| \leq |(kx+\beta)-(kx_i+\beta_i)|\leq 1.
		\]
		Since $\ell_{x\beta}$ and $\ell_{x_i\beta_i}$ both pass through the initial doorway
		of $H_k$, we know $|\beta-\beta_i|\leq 1$ and so $k|x-x_i|\leq 2$.  Thus, $|x-x_i|\leq 2/k$
		and so, since $k$ was arbitrary, $|x-x_i|\to 0$.
	\end{proof}

	The metric $\tilde d_S$ is equivalent to what we are used to in a metric on $\R$, but the
	metric $\tilde d_R$ is much stranger.

	\begin{proposition}
		The set $\R\backslash \Q$ of irrational numbers is closed with respect to $\tilde d_R$.
	\end{proposition}
	\begin{proof}
		Suppose $(x_i)$ is a sequence of irrational real numbers and $x\in\Q$.  Further,
		suppose $\tilde d_R(x_i,x)\to 0$.  This implies the existence of hallways
		$H^{(i)}\in s^{-1}(x_i)$ so that $H^{(i)}\to H\in s^{-1}(x)$ with respect to 
		$d_R$.  However, since $x\in \Q$, $H$ is a periodic infinite hallway, and so by
		Proposition \ref{PropQcClosed}, $H^{(i)}\not\to H$, a contradiction.
	\end{proof}

	The set $\R\backslash \Q$ is clearly not closed under $\tilde d_S$.  We now have
	an unusual situation.  The set $\Q$ is dense in $\R$ (its closure is $\R$) under both
	$\tilde d_S$ and $\tilde d_R$, however $\R\backslash \Q$ is dense in $\R$ under $\tilde d_S$,
	but not $\tilde d_R$.
	Stranger still, according to the following proposition, $\tilde d_R$ does not change very much.

	\begin{proposition}\label{PropdSdRSameOnQc}
		The metrics $\tilde d_R$ and $\tilde d_S$ are equivalent when restricted to the
		set $\R\backslash \Q$ of irrational numbers.
	\end{proposition}

	\begin{proof}
		First note that convergence in $d_R$ implies convergence in $d_S$
		since $d_R(H^{(i)},H)\to 0$ implies $|\comm(H^{(i)},H)|\to\infty$.  Thus
		convergence in $\tilde d_R$ implies convergence in $\tilde d_S$.

		Now, fix $\alpha\in \R\backslash \Q$ and $q\in\N$ and choose
		$\kappa>0$ so that the interval 
		\[
			B_\kappa(\alpha)=(\alpha-\kappa,\alpha+\kappa)\subset \R
		\] contains
		no rational points with denominator less than $q$.  Since $\tilde d_S$
		is equivalent to the standard metric on $\R$, there exists a $k$ so 
		that $\tilde d_S(\alpha,\gamma)<1/k$ implies $\gamma\in B_{\kappa}(\alpha)$.

		Now, if $H$ is an infinite hallway admitting a line of sight of slope
		$\alpha$, then $d_S(H,H')<1/k$ implies $d_R(H,H')\leq 1/q$.  Since $q$
		was arbitrary, if sequence of hallways converges to $H$ with respect
		to $d_S$, the same sequence converges to $H$ with respect to $d_R$.
		Thus, a sequence converging to $\alpha$ with respect to $\tilde d_S$ converges
		to $\alpha$ with respect to $\tilde d_R$.  This holds on all of $\R$ so long
		as $\alpha\in\R\backslash \Q$, and therefore it holds on all of $\R\backslash \Q$.
	\end{proof}

	We can also use $d_R$ to induce a metric on the set of sequences, $\Z^\N$, and in particular,
	the set of Sturmian sequences.  
	\begin{theorem}\label{ThmHallwaysEquivSturms}
		For an infinite hallway $H$, $\Phi(H)$ is a Sturmian sequence if and only
		if $H$ admits an infinite line of sight $\ell_{\alpha\beta}^\epsilon$ with
		$\alpha\in[0,1]$.
	\end{theorem}
	\begin{proof}
		Suppose $H$ is an infinite hallway with doorways $D_i$.
		Further, suppose $\Phi(H)$ is a Sturmian sequence.  Then $\Phi(H)=\Rf(\alpha,\beta)$
		or $\Phi(H)=\Rc(\alpha,\beta)$ for some $\alpha,\beta\in[0,1]\times \R$.  If
		$\Phi(H)=\Rf(\alpha,\beta)$ then, because we assume the initial doorway of $H$ 
		is $D_0=(0,1)$, we have 
		\[
			D_i = (\floor{i\alpha+\beta}, \floor{i\alpha+\beta}+1).
		\]
		Similarly, if $\Phi(H)=\Rc(\alpha,\beta)$,
		\[
			D_i = (\ceil{i\alpha+\beta}, \ceil{i\alpha+\beta}+1).
		\]
		In either case, $H$ admits the infinite line of sight $\ell_{\alpha\gamma}^\epsilon$ where
		$\gamma=\beta\mod 1$.

		Now suppose that $H$ admits the infinite line of sight $\ell_{\alpha\beta}^\epsilon$.
		By Proposition \ref{PropLittleNudge}, $(\ell_{\alpha\beta}+\epsilon)\cap H=\emptyset$ or
		$(\ell_{\alpha\beta}-\epsilon)\cap H=\emptyset$.  Suppose $(\ell_{\alpha\beta}+\epsilon)\cap H=\emptyset$.
		Then
		\[
			D_i = (\floor{i\alpha+\beta}, \floor{i\alpha+\beta}+1)
		\]
		and so $\Phi(H) = \Rf(\alpha,\beta)$.  Alternatively, suppose $(\ell_{\alpha\beta}-\epsilon)\cap H=\emptyset$.
		Then
		\[
			D_i = (\ceil{i\alpha+\beta}, \ceil{i\alpha+\beta}+1)
		\]
		and so $\Phi(H)=\Rc(\alpha,\beta)$.  In either case, $\Phi(H)$ is a rotation sequence and therefore
		a Sturmian sequence.
	\end{proof}

	Recall that $\Omega=\Phi(\{H:V(H)=1\})$.  In light of Theorem \ref{ThmHallwaysEquivSturms},
	$\mathcal S=\Omega\cap\{0,1\}^\N$ is the set of all Sturmian sequences.  $\mathcal S$ is $T$-invariant ($T(\mathcal S)=\mathcal S$),
	but it is not closed with respect to the standard metric.  Let $\hat d_R$ be the metric
	on sequences induced by $d_R$.  Again, $\hat d_R$
	induces the same topology on $\mathcal S$ as $d$, the standard metric on sequences,
	however under $\hat d_R$, $\mathcal S$ is closed, and the set of periodic Sturmian sequences is dense.


\bibliographystyle{abbrv}
\bibliography{JasonSiefkenDoorwaysProblem}

\footnotesize
%
%
%
%
%

\end{document}